\documentclass[12pt,reqno]{amsart}
\usepackage{amssymb}
\usepackage{mathrsfs,mathtools}
\usepackage{palatino}
\usepackage{bbm}
\usepackage{pictex}
\usepackage{amscd}
\usepackage{latexsym}
\usepackage{amssymb}
\usepackage{enumitem}
\usepackage{eucal}
\usepackage{MnSymbol}
\usepackage{eufrak}
\usepackage{verbatim}
\usepackage{bbm}
\usepackage{hyperref}
\hypersetup{colorlinks,linkcolor=blue,urlcolor=cyan,citecolor=red}
\usepackage{pictex}
\usepackage{amscd}
\usepackage{latexsym}
\usepackage{amssymb}
\usepackage{eucal}
\usepackage{eufrak}
\usepackage{verbatim}

\newcommand{\FF}{\mathbb{F}}

\newcommand{\NN}{\mathbb{N}}

\newcommand{\ZZ}{\mathbb{Z}}

\newcommand{\fgl}{\mathfrak{gl}}

\newcommand{\kk}{\mathbbm{k}}

\newcommand{\gl}{\mathfrak{gl}}
\newcommand{\fs}{\mathfrak{s}}

\newcommand{\yp}{{\rm Y}^{[p]}}
\newcommand{\zpl}{Z^{[p]}(\lambda(u))}

\DeclareMathOperator{\Char}{char}

\DeclareMathOperator{\ev}{ev}
\DeclareMathOperator{\Span}{Span}

\DeclareMathOperator{\id}{id}

\DeclareMathOperator{\Tab}{Tab}

\DeclareMathOperator{\Y}{Y}

\numberwithin{equation}{section}
\newtheorem{Theorem}{Theorem}[section]
\newtheorem{Lemma}[Theorem]{Lemma}

\newtheorem{Proposition}[Theorem]{Proposition}
\newtheorem{Remark}[Theorem]{Remark}

\theoremstyle{Theorem}

\newtheorem*{thm*}{Theorem}
\newtheorem*{thm**}{Corollary}
\newtheorem*{thm***}{Theorem B}
\theoremstyle{remark}

\newtheorem*{Definition}{Definition}

\numberwithin{equation}{section}

\newtheorem{theorem}{\textbf{Theorem}}[section]
\newtheorem{lemma}[theorem]{\textbf{Lemma}}

\newtheorem{proposition}[theorem]{\textbf{Proposition}}
\newtheorem{corollary}[theorem]{\textbf{Corollary}}

\numberwithin{equation}{section}

\begin{document}
	\title[Modular shifted super Yangian]{Representations of the modular shifted super Yangian $Y_{1|1}(\sigma)$}
	\author{Hao Chang, Ruiying Hou \lowercase{and} Hui Wu}
	\address[H. Chang]{School of Mathematics and Statistics,
		and Hubei Key Laboratory of Mathematical Sciences,
		Central China Normal University, Wuhan 430079, China}
	\email{chang@ccnu.edu.cn}
	\address[R. Hou]{School of Mathematics and Statistics,
		Central China Normal University, Wuhan 430079, China}
	\email{houry1998@163.com}
	\address[H. Wu]{School of Mathematics and Statistics, 
		Central China Normal University, Wuhan 430079, China}
	\email{1293108736@qq.com}
	
	
	\date{\today}
	
	\makeatletter
	
	\makeatother
	\maketitle
	\begin{abstract}
		Let $Y_{1|1}$ be the Yangian associated to the general linear Lie superalgebra $\mathfrak{gl}_{1|1}$,
		defined over an algebraically closed field $\mathbbm{k}$ of characteristic $p>2$.
		In this paper, 
		we classify the finite dimensional irreducible representations of the restricted 
		super Yangian $Y_{1|1}^{[p]}$ and the restricted truncated shifted super Yangian $Y_{1|1,\ell}^{[p]}(\sigma)$.
	\end{abstract}
	\section{Introduction}
	The Yangian $Y_{m|n}$ associated with the Lie superalgebra $\mathfrak{gl}_{m|n}$ over the complex
	field is a deformation of the universal enveloping algebra $U(\mathfrak{gl}_{m|n}[t])$ in the class of Hopf algebras.
	The original definition was given by Nazarov \cite{Na91} in terms of the {\it RTT presentation}. 
	A {\it Drinfeld-type presentation} corresponding to a standard
	Borel subalgebra was obtained by Gow \cite{Gow07},
	extending the results of Drinfeld \cite{Drin88} and
	Brundan and Kleshchev \cite{BK05} on the Yangian $Y_n$.
	The finite dimensional irreducible representations of $Y_{m|n}$ were classified
	by Zhang \cite{Zhang95, Zhang96}
	in a way similar to the representations of the Yangians associated with the simple Lie algebras, as given by Tarasov \cite{Ta85} and Drinfeld \cite{Drin88} (see also \cite{Mol07}).
	Recently,
	the finite dimensional irreducible 
	representations of the orthosymplectic super Yangians were further investigated
	in a series of papers \cite{Mol23-1, Mol23-2, Mol25, MR24}.
	
	In \cite{BBG13},
	Brown, Brundan and Goodwin introduced the {\it shifted super Yangians} $Y_{1|1}(\sigma)$ for $\fgl_{1|1}$. 
	This parallels the definition of {\it shifted Yangians} from \cite{BK06}.
	In particular,
	they established the connection between the {\it principal $W$-superalgebra} $W_{m|n}$ and the {\it truncated shifted super Yangian} $Y_{1|1,\ell}(\sigma)$.
	This allowed them to
	make an extensive study of the representation theory of $W_{m|n}$,
	see also \cite[Chapter 7]{BK08} for the classification of the finite dimensional irreducible representations of the shifted Yangians.
	Meanwhile,
	Peng \cite{Peng21} obtained a super analogue of the main result of \cite{BK06} for type $A$ Lie superalgebras in full generality.
	
	Let $\kk$ be an algebraically closed field of positive characteristic $p>0$.
	Brundan and Topley \cite{BT18} developed the theory of the shifted Yangian $Y_n(\sigma)$ over $\kk$.
	They defined the Yangian $Y_n$ over $\kk$ by the usual RTT presentation,
	and extended the {\it Drinfeld-type presentation} (\cite{Drin88, BK05}) from characteristic zero to characteristic $p$.
	In \cite{CH23},
	we use Nazarov's approach in the definition of the modular super Yangian and 
	gave the modular analogue of the Drinfeld-type presentation of the super Yangian $Y_{m|n}$.
	
	In the current article,
	we initiate a study of representations of the shifted super Yangian $Y_{1|1}(\sigma)$ over $\kk$.
	As a first step towards developing the modular representation theory of super Yangians,
	we first investigate the finite-dimensional irreducible representations of the restricted Yangian  $Y_{1|1}^{[p]}$ in detail.
	Over complex field,
	Zhang \cite{Zhang95} developed the representation theory of the super Yangian $Y_{1|1}$.
	However, when working over $\kk$, his approach cannot
	work in characteristic $p$.
	We classify the finite dimensional irreducible $Y_{1|1}^{[p]}$-modules by adapting the methods from \cite{CHT25} to the super setting.
	Motivated by the results of \cite{BBG13},
	we study the representations of the {\it restricted truncated shifted super Yangian} $Y_{1|1,\ell}^{[p]}(\sigma)$ further (see subsection \ref{section name:Shifted super Yangains}).
	As mentioned in Remark \ref{remark last},
	we expect that this result will be of importance in the classification of  certain irreducible represetations of the modular Lie superalgebra $\fgl_{m|n}$.
	
	The article is organized as follows.
	In Section \ref{secname : restricted super Yangians},
	we recall some preliminaries about the modular super Yangian $Y_{1|1}$ and its $p$-center.
	In particular, Theorem \ref{thm: yp hopf} shows that the Hopf structure descends to $Y_{1|1}^{[p]}$.
	There is an appendix at the end of the paper, 
	which provides a detailed proof of Proposition \ref{prop: delta send p to p}.
	In Section \ref{section name: rep of yp}, we define the baby Verma modules for $Y_{1|1}^{[p]}$ and prove that every finite dimensional irreducible representation is isomorphic to
	the simple head of some baby Verma module (Theorem \ref{thm: fd irr iso lplambda}).
	We also give the necessary and sufficient condition for an irreducible representation to be finite dimensional in terms of modular Drinfeld polynomials (Theorem \ref{theorem: lplambda fd condition}).
	Section \ref{section name: restricted shifted super Yangian} is devoted to further study about the irreducible representations of $Y_{1|1,\ell}^{[p]}(\sigma)$ (Theorem \ref{thm:classification of simple restricted ypi-module}).

	\bigskip
	\emph{Throughout this paper, $\kk$ denotes an algebraically closed field of characteristic $\Char(\kk)=:p>2$.}
	
	\section{Restricted Yangians}\label{secname : restricted super Yangians}
	\subsection{Modular Yangian $Y_{1|1}$}
	The super Yangian $Y_{m|n}$~associated to the general linear Lie superalgebra $\gl_{m|n}$ over the complex
	field was defined by Nazarov \cite{Na91} in terms of the {\it RTT presentation}.
	The {\it Drinfeld-type presentation} was found by Gow in \cite{Gow07}.
	In this paper, we only need the special case of $m=n=1$.
	For its definition, we fix a {\it standard parity} $|1|=0$ and $|2|=1$.
	Following \cite{Na91} (see also \cite{CH23}),
	the super Yangian associated to the general linear Lie superalgebra $\mathfrak{gl}_{1|1}$, denoted by $\Y:=Y_{1|1}$,
	is the associated superalgebra over $\kk$ with the RTT generators $\{ t_{i,j}^{(r)};1 \leq i,j \leq 2, r>0 \}$ subject the following relations:
	\begin{align}\label{def relation-coeff}
		[t_{i,j}^{(r)},t_{k,l}^{(s)}]=(-1)^{ \left|i\right| \left|j\right|+\left|i\right| \left|k\right|+\left|j\right| \left|k\right|} \sum_{t=0}^{{\rm min}\{r,s\}-1}(t_{k,j}^{(t)}t_{i,l}^{(r+s-1-t)}-t_{k,j}^{(r+s-1-t)}t_{i,l}^{(t)}),
	\end{align}
	where the parity $t_{i,j}^{(r)}$ is defined by $|i|+|j|~({\rm mod}~2)$, 
	and the bracket is understood as the supercommutator. 
	By convention, we set $t_{i,j}^{(0)}:=\delta_{i,j}$.
	
	The element $t_{i,j}^{(r)}$ is called an {\it even} ({\it odd}, respectively) element
	if its parity is $0$ ($1$, respectively). We define the formal power series
	\begin{align*}
		t_{i,j}(u):=\sum_{r \geq 0}t_{i,j}^{(r)}u^{-r} \in \Y[[u^{-1}]].
	\end{align*}
	Then these power series for all $1\leq i,j\leq 2$ can be collected together into a single matrix
	\begin{align*}
		T(u):=\begin{pmatrix}
			t_{1,1}(u) & t_{1,2}(u)\\
			t_{2,1}(u) & t_{2,2}(u)
		\end{pmatrix}.
	\end{align*}
	It is easily seen that, in terms of the generating series, the initial defining relation \eqref{def relation-coeff} may be rewritten as follows:
	\begin{align}\label{def relation-series}
		(u-v)[t_{i,j}(u),t_{k,l}(v)]=(-1)^{ \left|i\right| \left|j\right|+\left|i\right| \left|k\right|+\left|j\right| \left|k\right|} (t_{k,j}(u)t_{i,l}(v)-t_{k,j}(v)t_{i,l}(u)).
	\end{align}
	
	We need another set of generators for $\Y$ called {{\it Drinfeld generators}. 
		To define these, we consider the Gauss factorization 
		$$T(u) = F(u)D(u)E(u).$$
		This defines power series $d_i(u), e(u), f(u) \in \Y[[u^{-1}]]$ such that
		\begin{align*}
			D(u)=\begin{pmatrix}
				d_1(u) & 0\\
				0 & d_2(u)
			\end{pmatrix},
			E(u)=\begin{pmatrix}
				1 & e(u)\\
				0 & 1
			\end{pmatrix},
			F(u)=\begin{pmatrix}
				1 & 0\\
				f(u) & 1
			\end{pmatrix}
		\end{align*}
		Then we have that
		\begin{align}\label{t11t22 drinfeld}
			t_{1,1}(u)= d_1(u),\quad t_{2,2}(u)=f(u)d_1(u)e(u)+d_2(u),
		\end{align}
		\begin{align}\label{t12 t21 drinfeld}
			t_{1,2}(u)= d_1(u)e(u),\quad t_{2,1}(u)=f(u)d_1(u).
		\end{align}
		The Drinfeld generators are the elements $d_i^{(r)}, e^{(r)}$ and
		$f^{(r)}$ of $\Y$ defined from the expansions $d_i(u)=\sum_{r\geq 0}d_i^{(r)}u^{-r}, e(u)=\sum_{r> 0}e^{(r)}u^{-r}$ and $f(u)=\sum_{r>0}f^{(r)}u^{-r}$.
		Also define $d_i^{\prime (r)}$ from the identity $d_i^{\prime}(u)=\sum_{r\geq 0}d_i^{\prime (r)}u^{-r}:=d_i(u)^{-1}$.
		
		Under the assumption the characteristic $p>2$,
		the Drinfeld presentation of the modular Yangian $\Y$ is exactly the same as the one obtained by Gow over the complex field (see \cite[Theorem 3.8]{CH23} and \cite[Theorem 3]{Gow07}).
		The interested reader is referred to \cite{CH26} for the definition of the super Yangians in characteristic $2$.
		
		\begin{Theorem}\label{thm: Drinfeld presentation}
			The super Yangian $\Y$ is generated by the elements $\{d_i^{(r)}, d_i^{\prime (r)}; 1 \leq i \leq 2, r>0\}$ and $\{e^{(r)},f^{(r)}; r>0\}$ subject only to the following relations:
			\begin{align}	
				d_i^{(0)}=1,\quad \sum\limits_{t=0}^{r}d_i^{(t)}d_i^{\prime (r-t)}=\delta_{r,0};
			\end{align}
			\begin{align}\label{bracket dirdjs}
				[d_i^{(r)},d_j^{(s)}]=0;
			\end{align}
			\begin{align}\label{bracket dir esfs}
				[d_i^{(r)},e^{(s)}]=\sum\limits_{t=0}^{r-1}d_i^{(t)}e^{(r+s-1-t)},\quad [d_i^{(r)},f^{(s)}]=-\sum\limits_{t=0}^{r-1}f^{(r+s-1-t)}d_i^{(t)};
			\end{align}
			\begin{align}\label{erfs}
				[e^{(r)},f^{(s)}]=\sum\limits_{t=0}^{r+s-1}d_1^{\prime (t)}d_2^{(r+s-1-t)};
			\end{align}
			\begin{align}\label{bracket eres frfs}
				[e^{(r)},e^{(s)}]=[f^{(r)},f^{(s)}]=0.
			\end{align}	
		\end{Theorem}
		
		Here is the {\it PBW theorem} for $\Y$, see \cite[Theorem 3.11]{CH23}.
		\begin{Theorem}\label{theorem: PBW-drinfeld}
			Ordered supermonomials in the elements
			\begin{align}\label{drinfeld generators}
				\{d_i^{(r)};~1\leq i\leq 2,r>0\}\cup \{e^{(r)},f^{(r)};~r>0 \}
			\end{align}
			taken in any fixed ordering form a basis for $\Y$.
		\end{Theorem}
		\subsection{Restricted super Yangian $Y_{1|1}^{[p]}$}\label{subsection name:restricted yangian}
		We proceed to recall the description of the $p$-central elements of $\Y$ given in \cite{CH23}. 
		For $i = 1, 2$, we define
		\begin{equation}\label{def of biu}
			b_i(u):=\sum\limits_{r\geq 0}b_i^{(r)}u^{-r}=
			\begin{cases}
				d_1(u)d_1(u-1) \cdots d_1(u-p+1) & \quad \text{if}~ i=1,\\
				d_2(u)^{-1}d_2(u-1)^{-1} \cdots d_2(u-p+1)^{-1} & \quad \text{if}~ i=2.
			\end{cases}
		\end{equation}
		By \cite[Theorem 4.8]{CH23}
		the elements in
		\begin{align}\label{generators of p center}
			\{b_i^{(rp)};~1 \leq i \leq 2, r>0 \}
		\end{align}
		are algebraically independent, and lie in the center $Z(\Y)$ of $\Y$. 
		The subalgebra they generated is called {\it $p$-center} of $\Y$ and is denoted by $Z_p(\Y)$.
		According to \cite[Corollary 4.10]{CH23}, the super Yangian $\Y$ is free as a module over $Z_p(\Y)$ with basis
		given by the ordered supermonomials in the generators in \eqref{drinfeld generators} in which no exponent is $p$ or more for $d_i^{(r)}$, 
		we refer to such monomials as $p$-restricted supermonomials.
		
		We let $Z_p(\Y)_+$ be the maximal ideal of $Z_p(\Y)$ generated by the elements given in \eqref{generators of p center}. 
		The {\it restricted super Yangian} is defined by 
		\[
		\Y^{[p]}:=\Y/\Y Z_p(\Y)_+.    
		\]
		Clearly, the images in $\Y^{[p]}$ of the $p$-restricted supermonomials
		in the Drinfeld generators of $\Y$ form a basis of $\Y^{[p]}$. 
		When working with $\Y^{[p]}$ we often
		abuse notation by using the same symbols $d_i^{(r)}, e^{(r)}
		, f^{(r)}, t_{i,j}^{(r)}$ to refer to the elements of $\Y$
		and their images in $\Y^{[p]}$.
		\begin{Remark}\label{Remark: generators of p-center}
			In fact, the elements $\{b_i^{(r)};~i = 1, 2, r>0\}$ also belong to the $p$-center $Z_p(\Y)$ (see \cite[Lemma 4.5, Theorem 4.6]{CH23}). 
			Hence $b_i(u)$ is equal to $1$ in $\Y^{[p]}[[u^{-1}]]$.
		\end{Remark} 
		\subsection{Hopf algebra structure}\label{subsection:hopf}
		It is well known \cite[p. 125]{Na91} that $\Y$ is a Hopf superalgebra.
		Its comultiplication $\Delta$ and antipode $S$ are given by 
		\begin{align}\label{comul and antipode}
			\Delta(t_{i,j}(u))=\sum\limits_{k=1}^2t_{i,k}(u)\otimes t_{k,j}(u),\quad S(t_{i,j}(u))=t_{i,j}^{\prime}(u),
		\end{align}
		where $t_{i,j}^{\prime}(u)$ is the $(i,j)$-entry of the matrix $T(u)^{-1}$. The counit send $t_{i,j}(u)\mapsto \delta_{ij}$. 
		Now, we denote by $I_p :=\Y Z_p(\Y)_+$ the ideal defined in Subsection \ref{subsection name:restricted yangian}.
		In this section, we will prove that $I_p$ is a Hopf ideal of $\Y$. 
		Then the restricted super Yangian $\Y^{[p]}$ inherits from $\Y$ the Hopf algebra
		structure.
		
		The following Proposition may be viewed as the super version of \cite[Corollary 2.6]{CHT25}. We defer the proof to the Appendix \ref{section name:appendix}.
		
		\begin{Proposition}\label{prop: delta send p to p}
			For $i=1,2$,
			we have
			\begin{align*}
				\Delta(b_i(u))=b_i(u)\otimes b_i(u).
			\end{align*}
			In particular, $\Delta(I_p)\subseteq I_p \otimes \Y+\Y\otimes I_p$.
		\end{Proposition}
		
		Now we need to determine $S(b_i(u))$.
		For that, we introduce one more family of elements, we let
		\begin{equation*}
			b_{i}'(u)=\sum\limits_{r \geq 0}b_{i}'^{(r)}u^{(-r)}:=
			\begin{cases}
				t'_{1,1}(u)t'_{1,1}(u-1) \cdots t'_{1,1}(u-p+1) & \quad \text{if}~i=1,\\
				t_{2,2}(u)t_{2,2}(u-1) \cdots t_{2,2}(u-p+1) & \quad \text{if}~ i=2.
			\end{cases}
		\end{equation*}
		
		\begin{Proposition}\label{prop:S send p to p}
			We have $S(I_p)\subseteq I_p$.    
		\end{Proposition}
		\begin{proof}
			Since $S$ is an anti-antomorphism of $\Y$ \cite[p. 126]{Na91}, \eqref{t11t22 drinfeld} and \eqref{bracket dirdjs} imply that
			\[
			S(b_1(u))=S(t_{1,1}(u)t_{1,1}(u-1)\cdots t_{1,1}(u-p+1))=b_1'(u).
			\]
			Morwover, the Gauss decomposition yields $d_2(u)^{-1}=t_{2,2}'(u)$.
			By the same token, we have
			\[
			S(b_2'(u))=S(t_{2,2}(u)t_{2,2}(u-1)\cdots t_{2,2}(u-p+1))=b_2(u).
			\]
			Using the same proof of \cite[Lemma 5.6 and Theorem 5.7]{CH23}, one can show that 
			the $p$-center $Z_p(\Y)$ of $\Y$ can be generated by $\{b_1^{(r)},b_2'^{(r)};~r>0\}$
			or $\{b_1'^{(r)},b_2^{(r)};~r>0\}$. The assertion follows.
		\end{proof}
		The main result of this subsection reads:
		\begin{Theorem}\label{thm: yp hopf}
			The ideal $I_p$ is a hopf ideal of $\Y$.  
			In particular, $\Y^{[p]}$ inherits from $\Y$ the Hopf algebra
			structure.
		\end{Theorem}
		\begin{proof}
			This is a direct consequence of Proposition \ref{prop: delta send p to p} and Proposition \ref{prop:S send p to p}.   
		\end{proof}
		
		\subsection{Evaluation homomorphism}
		Let $U(\fgl_{1|1})$ be the universal enveloping algebra of the Lie superalgebra $\fgl_{1|1}$.
		There exists a surjective homomorphism
		\[
		\ev:\Y\rightarrow U(\fgl_{1|1})
		\]
		called the {\it evaluation homomorphism} (see for example \cite[(7)]{Gow07}),
		defined by
		\begin{align}\label{evaluation map}
			\ev(t_{i,j}(u)):=\delta_{i,j}+(-1)^{|i|}e_{i,j}u^{-1},   
		\end{align}
		where $e_{i,j}\in\fgl_{1|1}$ is the elementary matrix.
		Since we are in characteristic $p$,
		the Lie superalgebra $\fgl_{1|1}$ admits a natural structure of {\it restricted Lie superalgebra}.
		That is, the even subalgebra $(\fgl_{1|1})_{\bar{0}}$ is a restricted Lie algebra with $p$-map $(\fgl_{1|1})_{\bar{0}}\rightarrow (\fgl_{1|1})_{\bar{0}}$ sending $x\mapsto x^{[p]}$, which is defined on a basis of $(\fgl_{1|1})_{\bar{0}}$ by the rule $e_{i,i}^{[p]}=e_{i,i}$.
		General theory (cf. \cite[Section 2]{WZ09}) implies that 
		the element $e_{i,i}^p-e_{i,i}\in U(\fgl_{1|1})$ is central.
		Let $J_p$ be the ideal of $U(\fgl_{1|1})$ generated by the even central elements $\{e_{i,i}^p-e_{i,i};~i=1,2\}$.
		The quotient algebra $U^{[p]}(\fgl_{1|1}):=U(\fgl_{1|1})/J_p$ is called the {\it restricted enveloping algebra} of $\fgl_{1|1}$.
		
		\begin{Lemma}\label{Lemma: restricted ev hom}
			The evaluation homomorphism $\ev$ factors to a surjective algebra homomorphism 
			\[
			\ev^{[p]}: \yp\rightarrow U^{[p]}(\fgl_{1|1}).
			\]
		\end{Lemma}
		\begin{proof}
			It suffices to show that the images of the coefficients of $b_1(u)$ and $b_2'(u)$ under $\ev$ are contained in $J_p$. 
			Since $d_1(u)=t_{1,1}(u)$ \eqref{t11t22 drinfeld},
			we have
			\begin{align*}
				\ev(b_1(u))&=\ev((t_{1,1}(u)t_{1,1}(u-1)\cdots t_{1,1}(u-p+1))\\
				&=(1+e_{1,1}u^{-1})(1+e_{1,1}(u-1)^{-1})\cdots (1+e_{1,1}(u-p+1)^{-1})\\
				&=\frac{u+e_{1,1}}{u}\frac{u-1+e_{1,1}}{u-1}\cdots \frac{u-p+1+e_{1,1}}{u-p+1}.
			\end{align*}
			Write 
			\[
			\ev(b_1(u))=1+g^{(1)}u^{-1}+g^{(2)}u^{-2}+\cdots.
			\]
			Here the coefficient $g^{(r)}$ is a polynomial of $e_{1,1}$.
			For each $r>0$, the above expression implies that $g^{(r)}(i)=0$ for all $i\in\FF_p$.
			This yields $(e_{1,1}^p-e_{1,1})\mid g^{(r)}$. As a result, $g^{(r)}\in J_p$.
			One argues similarly for $b_2'(u)$.
		\end{proof}
		Using the evaluation homomorphism $\ev$, 
		we see that any representation of $\fgl_{1|1}$ can be regarded as a representation of $\Y$ by pullback, and any irreducible representation of $\fgl_{1|1}$ remains irreducible over $\Y$.
		By the same token, Lemma \ref{Lemma: restricted ev hom} 
		allows us to equip any $U^{[p]}(\fgl_{1|1})$-module with a structure of $\Y^{[p]}$-module.
		\section{Representations of $\yp$}\label{section name: rep of yp}
		In this section, we study the representations of the restricted super Yangian $\yp$.
		To describe the finite dimensional irreducible modules of $\yp$,
		we need to construct analogues of highest weight
		representations in characteristic zero (see \cite[Section III]{Zhang95}).
		\subsection{Baby Verma modules}
		We define
		\[
		{\bf I}_{\NN}:=\{(i_1,i_2,\dots);~i_k\in\mathbb{Z}_{\geq 0}~\text{and only finitely many are non-zero}\}.
		\]
		Given $I=(i_1,i_2,\cdots)\in {\bf I}_{\NN}$, we let
		\[
		f^I:=\prod\limits_{r>0}(f^{(r)})^{i_r}   
		\]
		and similarly we define elements $d_1^I,d_2^I$ and $e^I$.
		Consider two subsets of ${\bf I}_{\NN}$
		\[
		{\bf I}_0:=\{(i_1,i_2,\dots)\in {\bf I}_{\NN};~0\leq i_k<p\}\quad {\bf I}_1:=\{(i_1,i_2,\dots)\in {\bf I}_{\NN};~0\leq i_k<2\},
		\]
		so that 
		\begin{align}\label{PBW for Yp}
			\{f^{I_1}d_1^{I_2}d_2^{I_3}e^{I_4};~I_2,I_3\in {\bf I}_0, I_1,I_4\in {\bf I}_1\}
		\end{align}
		is the PBW basis of $\yp$ (see Section \ref{subsection name:restricted yangian}).
		
		\begin{Definition}
			A formal power series $f(u)\in\kk[[u^{-1}]]$ is called {\it restricted}, provided
			\[
			f(u)f(u-1)\cdots f(u-p+1)=1,
			\] 
			which implies that $f(u)\in 1+u^{-1}\kk[[u^{-1}]]$.
			Given two formal series
			\[
			\lambda_i(u)=1+\lambda_i^{(1)}u^{-1}+\lambda_i^{(2)}u^{-1}+\cdots,\quad i=1,2,   
			\]
			we say that the tuple $\lambda(u):=(\lambda_1(u),\lambda_2(u))$ is restricted if both $\lambda_1(u)$ and $\lambda_2(u)$ are restricted.
		\end{Definition}
		
		Let $\lambda(u)=(\lambda_1(u),\lambda_2(u))$ be a restricted tuple. 
		We define the {\it baby Verma module} $\zpl$ corresponding to $\lambda(u)$ as the quotient of $\yp$ by the left ideal generated by the elements $e^{(r)}$ with $r>0$, and by $d_i^{(r)}-\lambda_i^{(r)}$ with $i=1,2$ and $r>0$.
		Denote by $1_{\lambda(u)}$ the image of the element $1\in\yp$ in the quotient.
		Clearly, $\zpl$ is a cyclic module and $\zpl=\yp.1_{\lambda(u)}$.
		We put $\Y^{[p],-}:=\Span_{\kk}\{f^I;~I\in {\bf I}_1\}$ and $(\Y^{[p],-})_+:=\Span_{\kk}\{f^I;~I\in {\bf I}_1, I\neq (0,0,\dots)\}$.
		Owing to \eqref{PBW for Yp}, there is an isomorphism
		\[
		\zpl\cong\Y^{[p],-}\otimes_{\kk}\kk 1_{\lambda(u)}   
		\]
		of vector spaces.
		
		Alternatively, the baby Verma modules can be described in terms of the RTT presentation of $\Y$ (cf. \cite[Proposition 3.2.2]{Mol07}). 
		
		\begin{Proposition}\label{prop: zpl defind by rtt}
			Let $\lambda(u)=(\lambda_1(u),\lambda_2(u))$ be restricted.
			The baby Verma module $Z^{[p]}(\lambda(u))$ equals to the quotient of $\yp$ by the left ideal $J$ which is generated by the elements $t_{1,2}^{(r)}$ with $r>0$ and by
			by $t_{i,i}^{(r)}-\lambda_i^{(r)}$ with $i=1, 2$ and $r>0$.
		\end{Proposition}
		\begin{proof}
			The proof, which uses \eqref{t11t22 drinfeld} and \eqref{t12 t21 drinfeld}, is similar to the proof of \cite[Proposition 3.3]{CHT25}, and will be skipped here.
		\end{proof}
		
		\begin{Proposition}\label{prop: zpl unique max submodule}
			$\zpl$ has a unique maximal submodule.    
		\end{Proposition}
		\begin{proof}
			We shall prove that any proper submodule $M$ of $\zpl$ is contained in the vector space $(\Y^{[p],-})_+.1_{\lambda(u)}$.
			Suppose that $M\nsubseteq (\Y^{[p],-})_+.1_{\lambda(u)}$, then there is a non-zero element $y=(1-x).1_{\lambda(u)}\in M$ and $x\in(\Y^{[p],-})_+$.
			Now \eqref{bracket eres frfs} implies that $x$ is nilpotent. 
			Assume that $x^n=0$ for some positive integer $n$. 
			We have for $x':=1+x+\cdots+x^{n-1}$, $x'y=1_{\lambda(u)}\in M$,
			so that $M=\zpl$, a contradiction. On the other hand, it is obvious that $(\Y^{[p],-})_+.1_{\lambda(u)}$ is a proper subspace $\zpl$.
			Hence, the sum of all proper submodules of $\zpl$ is the unique
			maximal submodule.
		\end{proof}
		
		Let $L$ be a finite dimensional representation of $\yp$.
		We define the subspace $L^0$ of $L$ via
		\[
		L^0:=\{v\in L;~e(u).v=0\}.
		\]

		\begin{Lemma}\label{Lemma: L0 nonzero}
			If $L$ is a finite dimensional representation of $\yp$,
			then the space $L^0$ is non-zero.
		\end{Lemma}
		\begin{proof}
			Suppose that $L^0=(0)$.
			Since $(e^{(1)})^2=0$ \eqref{bracket eres frfs}, there exists a non-zero vector $v_1\in L$ such that $e^{(1)}.v_1=0$. 
			By assumption, we can find a positive integer $n_1$ such that $e^{(r)}.v_1=0$ for all $1\leq r\leq n_1$ and $e^{(n_1+1)}.v_1\neq 0$.
			Setting $v_2:=e^{(n_1+1)}.v_1$, we have $e^{(n_1+1)}.v_2=(e^{(n_1+1)})^2.v_1=0$.
			For $1\leq i\leq n_1$, again the relation \eqref{bracket eres frfs} yields
			\[
			e^{(i)}.v_2=e^{(i)}e^{(n_1+1)}.v_1=-e^{(n_1+1)}e^{(i)}.v_1=0.
			\]
			This shows that $e^{(r)}.v_2=0$ for all $1\leq r\leq n_1+1$, 
			so that we can find a positive integer $n_2>n_1$ such that $e^{(r)}.v_2=0$ for all $1\leq r\leq n_2$ and $e^{(n_2+1)}.v_2\neq 0$.
			
			Repeat the above argument, we obtain two sequences $v_1,v_2,\dots$ and $1\neq n_1<n_2<\cdots$ which satisfy (i) $v_i\neq 0$; (ii) $e^{(r)}.v_i=0$ for all $1\leq r\leq n_i$; (iii) $e^{(n_i+1)}v_i\neq 0$ for every positive integer $i$. Let $m$ be an arbitrary positive integer. 
			Assume that $\sum_{i=1}^m c_iv_i=0$. It follows from (ii) that $c_1e^{(n_1+1)}.v_1=0$, and property (iii) gives $c_1=0$.
			By the same token, we obtain all $c_i=0$.
			Therefore $v_1,v_2,\dots,v_m$ are linearly independent.
			As $L$ is finite dimensional, we arrive at a contradiction.
		\end{proof}
		
		By Proposition \ref{prop: zpl unique max submodule}, 
		we know that $\zpl$ has a unique simple quotient, denoted by $L^{[p]}(\lambda(u))$. 
		
		Now we state the main theorem of this section, which is the modular analogue of \cite[Theorem 2 and Theorem 3]{Zhang95}.
		\begin{Theorem}\label{thm: fd irr iso lplambda}
			Every finite-dimensional simple $\yp$-module $L$ is isomorphic to some $L^{[p]}(\lambda(u))$.
		\end{Theorem}
		\begin{proof}
			By Lemma \ref{Lemma: L0 nonzero}, we know that $L^0\neq (0)$.
			Next we show that the subsapce $L^0$ is invariant with respect to the action of all elements $d_i^{(r)}$. 
			If $v\in L^0$, then \eqref{bracket dir esfs} implies that $[d^{(r)},e^{(s)}].v=0$ for all $r,s>0$. It follows that $e^{(s)}d_i^{(r)}.v=0$. As a result, $d_i^{(r)}.v\in L^0$.
			
			Furthermore, \eqref{bracket dirdjs} implies that the elements $d_i^{(r)}$ with $i=1,2$ and $r>0$ act on $L^0$ as pairwise commuting operators.
			Hence, there exists a non-zero vector $\zeta\in L^0$ such that $d_i^{(r)}.\zeta=\lambda_i^{(r)}\zeta$, where $\lambda_i^{(r)}\in\kk$.
			Letting $\lambda_i(u):=1+\sum_{r\geq 1}\lambda_i^{(r)}u^{-r}\in\kk[[u^{-1}]]$, we put $\lambda(u):=(\lambda_1(u),\lambda_2(u))$.
			Then $d_i(u).\zeta=\lambda_i(u)\zeta$.
			Since $1=b_i(u)=d_i(u)d_i(u-1)\cdots d_i(u-p+1)$ in $\yp$ (Remark \ref{Remark: generators of p-center}),
			it follows that $\lambda(u)$ is restricted.
			As $L$ is irreducible, there results a surjective homomorphism $Z^{[p]}(\lambda(u))\twoheadrightarrow L;~1_{\lambda(u)}\mapsto \zeta$.
			Proposition \ref{prop: zpl unique max submodule} now yields an isomorphism $L\cong L^{[p]}(\lambda(u))$.
		\end{proof}
		\subsection{Finite dimensionality conditions}
		For any $\lambda_1,\lambda_2\in \FF_p$, we consider the restricted irreducible $\fgl_{1|1}$-module $L(\lambda_1,\lambda_2)$. 
		The module $L(\lambda_1,\lambda_2)$ is generated by a vector $\xi$ and 
		\begin{align}\label{eij act on xi}
			e_{i,i}.\xi=\lambda_i\xi;~i=1,2, \quad e_{1,2}.\xi=0.
		\end{align}
		Since 
		\begin{align}\label{e12e21xi}
			e_{1,2}e_{2,1}.\xi=[e_{1,2},e_{2,1}].\xi=(e_{1,1}+e_{2,2}).\xi=(\lambda_1+\lambda_2)\xi,   
		\end{align}
		the module $L(\lambda_1,\lambda_2)$ is two-dimensional if $\lambda_1+\lambda_2\neq 0$ and one-dimensional otherwise.
		Lemma \ref{Lemma: restricted ev hom} allows us to view $L(\lambda_1,\lambda_2)$ as a $\yp$-module. 
		We will keep the same notation $L(\lambda_1,\lambda_2)$ for this $\yp$-module and call it the {\it evaluation module}.
		By Proposition \ref{prop: zpl defind by rtt},
		$L(\lambda_1,\lambda_2)$ is isomorphic to the module $L^{[p]}(\lambda_1(u),\lambda_2(u))$, where $\lambda_1(u)=1+\lambda_1u^{-1}$ and $\lambda_2(u)=1-\lambda_2u^{-1}$.
		
		Since $\yp$ is a Hopf quotient of $\Y$ (Theorem \ref{thm: yp hopf}), the comultiplication of $\Y$ descends to $\yp$, and we will use the same symbol $\Delta$ for the comultiplication of $\yp$.
		If $L$ and $M$ are any two $\yp$-modules, then the tensor product $L\otimes M$ can be equipped with a $\yp$-action with the use of the comultiplication $\Delta$ on $\yp$ by the rule
		\[
		x.(l\otimes m):=\Delta(x)(l\otimes m),\quad x\in\yp, l\in L, m\in M.
		\]
		
		Denote by $\tau$ the anti-automorphism of the algebra $\Y$, defined by
		\begin{align}\label{tau-antiauto}
			\tau:\Y\rightarrow \Y;~t_{i,j}(u)\rightarrow t_{3-i,3-j}(u).   
		\end{align}
		This is the composition of the anti-automorphism (\cite[(2.15)]{Na20}) and automorphism (\cite[Lemma 1]{Gow07}) of $\Y$.
		By Section \ref{subsection:hopf}, $\tau$ induces the anti-automorphism of the algebra $\yp$, which we again denote by $\tau$.
		Given any finite dimensional $\yp$-module $L$, we will denote by $L^*$ the dual vector space whose elements are linear maps $\omega:L\rightarrow\kk$. Equip $L^*$ with an $\yp$-module structure by setting
		\[
		(y.\omega)(\eta):=\omega(\tau(y).\eta)\quad \text{for}\quad y\in\yp\quad \text{and}\quad \omega\in L^*, \eta\in L.
		\]
		Using the same argument as in \cite[p. 110]{Mol07} one immediately verifies that $L(\lambda_1,\lambda_2)^*$ is also irreducible.
		Let us consider a two-dimensional irreducible module $L(\lambda_1,\lambda_2)$.
		The module $L(\lambda_1,\lambda_2)^*$ can be generated by the vector $\xi^*$ defined by $\xi^*(\xi)=1$ and $\xi^*(e_{2,1}.\xi)=0$. 
		Direct computation shows that 
		\[
		e_{1.2}.\xi^*=0, e_{1,1}.\xi^*=-\lambda_2\xi^*, e_{2,2}.\xi^*=-\lambda_1\xi^*.
		\]
		Observing \eqref{eij act on xi}, we thus obtain $L(\lambda_1,\lambda_2)^*\cong L(-\lambda_2,-\lambda_1)$.
		
		Now let $\lambda_1(u)$ and $\lambda_2(u)$ be two restricted polynomials in $u^{-1}$ of degree not more than $k$.
		Write the decompositions
		\begin{align*}
			\lambda_1(u):=(1+\lambda_1^{(1)}u^{-1})\cdots(1+\lambda_1^{(k)}u^{-1}),\\
			\lambda_2(u):=(1-\lambda_2^{(1)}u^{-1})\cdots(1-\lambda_2^{(k)}u^{-1}).
		\end{align*}
		Thanks to \cite[Lemma 3.2]{CHT25}, we have all the constants $\lambda_i^{(r)}$ belong to the finite field $\FF_p$.
		
		The following result is the modular analogue \cite[Theorem 5]{Zhang95}.
		Our proof is adapted from \cite[Proposition 3.7]{CHT25} (see also \cite[Proposition 3.3.2]{Mol07}).
		Since our setting is different from that given in \cite[p. 112]{Mol07}, 
		we provide a detailed proof here.
		\begin{Proposition}\label{Prop:tensor product module irre}
			Suppose that $\lambda_1^{(i)}+\lambda_2^{(j)}\neq 0$ for all $1\leq i,j\leq k$. Then the representation $L^{[p]}(\lambda_1(u),\lambda_2(u))$ of $\yp$ is isomorphic to the tensor product module
			\begin{align}\label{tensor product module}
				L(\lambda_1^{(1)},\lambda_2^{(1)})\otimes L(\lambda_1^{(2)},\lambda_2^{(2)})\otimes \cdots \otimes L(\lambda_1^{(k)},\lambda_2^{(k)}).
			\end{align}
		\end{Proposition}
		\begin{proof}
			Denote the module \eqref{tensor product module} by $L$. 
			Let $\xi_i$ be the generating vector of $L(\lambda_1^{(i)},\lambda_2^{(i)})$ for $i=1,\dots, k$.
			Using the definition \eqref{comul and antipode} of $\Delta$,
			we obtain $t_{1,2}(u).\xi=0$ and $t_{i,i}(u).\xi=\lambda_i(u)\xi$,
			where $\xi=\xi_1\otimes\cdots\otimes \xi_k$.
			The proposition will follow if we prove that the module $L$ is irreducible.
			
			We claim that any vector $\zeta\in L$ satisfying $t_{1,2}(u).\zeta=0$ is proportional to $\xi$. Now proceed by induction on $k$.
			By assumption, $\lambda_1^{(1)}+\lambda_2^{(1)}\neq 0$.
			The case $k=1$ immediately follows
			from \eqref{e12e21xi}.
			Suppose that $k\geq 2$.
			Write any such vector 
			\[
			\zeta=\xi_1\otimes \zeta_1+e_{2,1}.\xi_1\otimes \zeta_2,\quad \text{where}~\zeta_r\in L(\lambda_1^{(2)},\lambda_2^{(2)})\otimes \cdots \otimes L(\lambda_1^{(k)},\lambda_2^{(k)}),~r=1,2.
			\]
			We first assume that $\zeta_2\neq 0$.
			Applying $t_{1,2}(u)$ to $\zeta$, we get
			\begin{align}\label{proof zhongde formula-1}
				t_{1,2}(u)e_{2,1}.\xi_1\otimes t_{2,2}(u).\zeta_2+t_{1,1}(u).\xi_1\otimes t_{1,2}(u).\zeta_1+t_{1,1}(u)e_{2,1}.\xi_1\otimes t_{1,2}(u).\zeta_2=0      
			\end{align}
			Using the definition of the Yangian action on $L(\lambda_1^{(1)},\lambda_2^{(1)})$ in conjunction with \eqref{e12e21xi}, we obtain
			\[
			t_{1,2}(u)e_{2,1}.\xi_1= u^{-1}e_{1,2}e_{2,1}.\xi_1=u^{-1}(\lambda_1^{(1)}+\lambda_2^{(1)})\xi_1,
			\]
			and
			\[
			t_{1,1}(u)e_{2,1}.\xi_1=(1+e_{1,1}u^{-1})e_{2,1}.\xi_1=(1+(\lambda_1^{(1)}-1)u^{-1})e_{2,1}.\xi_1.
			\]
			Hence, taking the coefficient of $e_{2,1}.\xi_1$ in \eqref{proof zhongde formula-1} gives
			\[
			(1+(\lambda_1^{(1)}-1)u^{-1})t_{1,2}(u).\zeta_2=0,  
			\]
			implying the relation $t_{1,2}(u).\zeta_2=0$.
			The induction hypothesis implies that $\zeta_2$ must be proportional to $\xi_2\otimes\cdots\otimes \xi_k$.
			Therefore, we obtain by \eqref{comul and antipode} and \eqref{evaluation map} that
			\begin{align}\label{proof zhongde formula-2}
				t_{2,2}(u).\zeta_2=(1-\lambda_2^{(2)}u^{-1})\cdots ((1-\lambda_2^{(k)}u^{-1}))\zeta_2.
			\end{align}
			Then taking the coefficient of $\xi_1$ in \eqref{proof zhongde formula-1} we derive
			\[
			(1+\lambda_1^{(1)}u^{-1})t_{1,2}(u).\zeta_1+u^{-1}(\lambda_1^{(1)}+\lambda_2^{(1)})t_{2,2}(u).\zeta_2=0.
			\]
			Hence, multiplying by $u^k$ and taking into account \eqref{proof zhongde formula-2} we get
			\[
			(u+\lambda_1^{(1)})u^{k-1}t_{1,2}(u).\zeta_1+(\lambda_1^{(1)}+\lambda_2^{(1)})(u-\lambda_2^{(2)})\cdots (u-\lambda_2^{(k)})\zeta_2=0.
			\]
			By taking the value $u=-\lambda_1^{(1)}$ we obtain the relation
			\[
			(\lambda_1^{(1)}+\lambda_2^{(1)})(\lambda_1^{(1)}+\lambda_2^{(2)})\cdots (\lambda_1^{(1)}+\lambda_2^{(k)})=0.  
			\]
			But this contradicts the conditions. 
			Thus, $\zeta_2=0$.
			That is, $\zeta=\xi_1\otimes \zeta_1$. 
			The above argument in conjunction with induction hypothesis prove the claim.
			
			Suppose now that $M$ is a nonzero submodule of $L$.
			Acccording to Lemma \ref{Lemma: L0 nonzero} and \eqref{t12 t21 drinfeld},
			$M$ must contain a nonzero vector $\zeta$ such that $t_{1,2}(u).\zeta=0$.
			The foregoing observation implies that $M$ contains the vector $\xi$.
			It remains to prove that the cyclic span $K:=\yp.\xi$ coincides with $L$.
			Recall that we already have $L(\lambda_1^{(i)},\lambda_2^{(i)})^*\cong L(-\lambda_2^{(i)},-\lambda_1^{(i)})$.
			Observe that $\Delta\circ\tau=(\tau\otimes\tau)\circ\Delta$,
			it is easy to verify that (cf. \cite[Proposition 3.2.12]{Mol07}) the dual module $L^*$ is isomorphic to the tensor product
			\[
			L(-\lambda_2^{(1)},-\lambda_1^{(1)})\otimes L(-\lambda_2^{(2)},-\lambda_1^{(2)})\otimes \cdots \otimes L(-\lambda_2^{(k)},-\lambda_1^{(k)}).  
			\]
			As before, the generating vector of $L(-\lambda_2^{(i)},-\lambda_1^{(i)})\cong L(\lambda_1^{(i)},\lambda_2^{(i)})^*$ will be denoted by $\xi_i^*$.
			Remember that $\xi_i^*(\xi_i)=1$ and $\xi_i^*(e_{2,1}.\xi_i)=0$.
			Suppose now that the $K\subsetneq L$ and consider its annihilator
			\[
			{\rm Ann} K:=\{\omega\in L^*;~\omega(\eta)=0~\text{for all}~\eta\in K\}.
			\]
			Then ${\rm Ann} K$ is a nonzero submodule of $L^*$,
			which does not contain the vector $\xi_1^*\otimes\cdots\otimes\xi_k^*$ (see also \cite[p. 564]{Mol23-1}).
			However, this contradicts the claim verified in the first part of the proof,
			because the conditions on the parameters $\lambda_1^{(i)}$ and $\lambda_2^{(i)}$ stated in the theorem will remain satisfied after we replace each $\lambda_1^{(i)}$ by $-\lambda_2^{(i)}$ and each $\lambda_2^{(i)}$ by $-\lambda_1^{(i)}$.
		\end{proof}
		
		For any power series $f(u)\in 1+u^{-1}\kk[[u^{-1}]]$,
		it follows from the defining relation \eqref{def relation-series} (see also \cite[Section  3.7]{CH23}) that there is an automorphism defined via
		\begin{align}\label{auto:Multiplication by a power series}
			\mu_f:\Y\rightarrow \Y;~t_{i,j}(u)\mapsto f(u)t_{i,j}(u).  
		\end{align}
		\begin{Lemma}\label{lemma: flambdai are poly}
			If $L^{[p]}(\lambda(u))=L^{[p]}(\lambda_1(u),\lambda_2(u))$ is finite dimensional, then there exists a formal series $f(u)\in 1+u^{-1}\kk[[u^{-1}]]$ such that $f(u)\lambda_1(u)$ and $f(u)\lambda_2(u)$ are polynomials in $u^{-1}$.
		\end{Lemma}
		\begin{proof}
			The proof is the same as that of \cite[Proposition 3.3.1]{Mol07}, and will be omitted here. 
		\end{proof}
		\begin{Remark}\label{remark: f-lambdai poly not ness restricted}
			Let $f(u)\in 1+u^{-1}\kk[[u^{-1}]]$ be a formal power series.
			If the automorphism $\mu_f$ \eqref{auto:Multiplication by a power series}  factors to an automorphism of $\yp$, then clearly $f$ must be restricted.
		\end{Remark}
		\begin{Theorem}\label{theorem: lplambda fd condition}
			The irreducible representation $L^{[p]}(\lambda(u))=L^{[p]}(\lambda_1(u),\lambda_2(u))$ is finite dimensional if and only if there exists monic polynomials $P_1(u)$ and $P_2(u)$ in $u$ such that
			\begin{align}\label{Drinfeld poly}
				\frac{\lambda_1(u)}{\lambda_2(u)}=\frac{P_1(u)}{P_2(u)},
			\end{align}
			where the polynomials $P_1(u)$ and $P_2(u)$ are of the same degree and have no common roots.
		\end{Theorem}
		\begin{proof}
			Suppose that the representation $L^{[p]}(\lambda_1(u),\lambda_2(u))$ is finite dimensional. Then by Lemma \ref{lemma: flambdai are poly} we can find a formal series $f(u)$ such that
			\begin{align*}
				f(u)\lambda_1(u)=(1+\lambda_1^{(1)}u^{-1})\cdots(1+\lambda_1^{(k)}u^{-1}),\\
				f(u)\lambda_2(u)=(1+\lambda_2^{(1)}u^{-1})\cdots(1+\lambda_2^{(k)}u^{-1}),
			\end{align*}
			for some $k\geq 0$ and some scalars $\lambda_i^{(r)}$.
			It follows that
			\[
			\frac{\lambda_1(u)}{\lambda_2(u)}=\frac{f(u)\lambda_1(u)}{f(u)\lambda_2(u)}=\frac{(1+\lambda_1^{(1)}u^{-1})\cdots(1+\lambda_1^{(k)}u^{-1})}{(1+\lambda_2^{(1)}u^{-1})\cdots(1+\lambda_2^{(k)}u^{-1})}.
			\]
			Without loss of generality, we may
			assume further that $\lambda_1^{(i)}\neq \lambda_2^{(j)}$ for any $1\leq i,j\leq k$.
			Then the polynomials 
			\[
			P_1(u):=(u+\lambda_1^{(1)})\cdots(u+\lambda_1^{(k)})\quad \text{and}\quad P_2(u):=(u+\lambda_2^{(1)})\cdots(u+\lambda_2^{(k)})
			\]
			satisfy \eqref{Drinfeld poly}.
			
			Conversely, suppose \eqref{Drinfeld poly} holds for polynomials 
			\[
			P_1(u)=(u+\mu_1^{(1)})\cdots(u+\mu_1^{(s)})\quad \text{and}\quad P_2(u)=(u-\mu_2^{(1)})\cdots(u-\mu_2^{(s)}).
			\]
			Set 
			\begin{align*}
				\mu_1(u):=(1+\mu_1^{(1)}u^{-1})\cdots(1+\mu_1^{(s)}u^{-1}),\\
				\mu_2(u):=(1-\mu_2^{(1)}u^{-1})\cdots(1-\mu_2^{(s)}u^{-1}).
			\end{align*}
			By construction we have
			\begin{align}\label{useful formula-1}
				\frac{\mu_1(u)}{\mu_2(u)}=\frac{P_1(u)}{P_2(u)}.  
			\end{align}
			For each $1\leq r\leq s$, we consider the irreducibe $\fgl_{1|1}$-module $L(\mu_1^{(r)},\mu_2^{(r)})$ which is generated by $\xi_r$ and the module structure is given by
			\begin{align}\label{proof zhong de 2d module}
				e_{i,i}.\xi_r=\mu_i^{(r)}\xi_r, i=1,2,\quad e_{1,2}.\xi_r=0.
			\end{align}
			Since $\mu_1^{(r)}+\mu_2^{(r)}\neq 0$ by assumption,
			the module $L(\mu_1^{(r)},\mu_2^{(r)})$ is two-dimensional with basis $\{\xi_r,e_{2,1}.\xi_r\}$.
			We can regard them as $\Y$-modules via \eqref{evaluation map} consider the tensor product module
			\[
			L:=L(\mu_1^{(1)},\mu_2^{(1)})\otimes\cdots\otimes L(\mu_1^{(s)},\mu_2^{(s)}). 
			\]
			Obviously, this module is finite dimensional.
			We put $\xi:=\xi_1\otimes\cdots\otimes\xi_s$.
			Using the comultiplication \eqref{comul and antipode} in conjunction
			with \eqref{proof zhong de 2d module} one can show by direct computation that
			\[
			t_{i,i}(u).\xi=\mu_i(u)\xi, i=1,2, \quad t_{1,2}(u).\xi=0.
			\]
			Let $M:=\Y.\xi\subseteq L$ be the cyclic span of $L$.
			Twisting the action of $\Y$ on $M$ by the automorphism \eqref{auto:Multiplication by a power series} with $f(u)=\mu_2(u)^{-1}$,
			we obtain a $\Y$-module which will be denoted by $\widetilde{M}$.
			Clearly, $\widetilde{M}$ is also generated by $\xi$.
			By abuse of notation, we still write $\widetilde{M}=\Y.\xi$.
			Thanks to \eqref{Drinfeld poly} and \eqref{useful formula-1}, we have
			\[
			t_{1,1}(u).\xi=\frac{\mu_1(u)}{\mu_2(u)}\xi=\frac{\lambda_1(u)}{\lambda_2(u)}\xi,\quad t_{2,2}(u).\xi=\xi,\quad t_{1,2}(u).\xi=0.
			\]
			Notice that $\lambda_1(u)/\lambda_2(u)$ is restricted.
			This implies that $\widetilde{M}=\Y.\xi=\yp.\xi$,
			and there results a surjective homomorphism 
			\[
			Z^{[p]}(\frac{\lambda_1(u)}{\lambda_2(u)},1)\twoheadrightarrow\widetilde{M}=\yp.\xi.
			\]
			As a result, $L^{[p]}(\frac{\lambda_1(u)}{\lambda_2(u)},1)$ is finite dimensional.
			Since $\lambda_2(u)$ is restricted, we apply again the twisted action of $\yp$ on $L^{[p]}(\frac{\lambda_1(u)}{\lambda_2(u)},1)$ with $f(u)=\lambda_2(u)$ (see Remark \ref{remark: f-lambdai poly not ness restricted}) to obtain the module $L^{[p]}(\lambda_1(u),\lambda_2(u))$.
			We conclude that the module
			$L^{[p]}(\lambda_1(u),\lambda_2(u))$ is also finite dimensional.
		\end{proof}
		\begin{Remark}
			Suppose that $Q_1(u)$ and $Q_2(u)$ are another two co-prime polynomials and 
			\[
			\frac{P_1(u)}{P_2(u)}=\frac{Q_1(u)}{Q_2(u)}.
			\]
			This means that $P_1(u)Q_2(u)=P_2(u)Q_1(u)$.
			As $P_1$ and $P_2$ have no common roots, 
			it follows immediately that $P_i(u)=Q_i(u)$ for $i=1,2$,
			so that the monic polynomials occurring in the above theorem are unique.
			We will refer them as the modular Drinfeld polynomials of the representations, cf. \cite{Drin88}. 
			This is somewhat different from the modular Yangian case (\cite[Remark, p. 6985]{Kal20} and \cite[Remark 3]{CHT25}).
		\end{Remark}
		\section{Restricted shifted super Yangians}\label{section name: restricted shifted super Yangian}
		Over the field of complex numbers, Brown, Brundan and Goodwin \cite{BBG13} introduced the {\it (truncated) shifted super Yangians} for $\fgl_{1|1}(\mathbb{C})$ which parallel the definition of shifted Yangians from \cite{BK06} (Actually, this can be defined for all $\fgl_{m|n}$, cf. \cite[Section 4]{Peng21} for example).
		In this section, we will study the modular representations of (restricted) truncated shifted Yangians for $\fgl_{1|1}$.
		\subsection{Shifted super Yangains}\label{section name:Shifted super Yangains}
		Let $\sigma=(s_{i,j})_{1\leq i,j\leq 2}$ be a $2\times 2$ matrix of nonnegative integers with $s_{i,i}=0$ for $i=1,2$.
		We refer to such a matrix as a {\it shift matrix}.
		Following \cite{BBG13} (see also \cite{BK06}),
		the {\it shifted super Yangian} associated to the matrix $\sigma$, denoted by $\Y_\sigma$,
		is the superalgebra with even generators $\{d_i^{(r)};~i=1,2, r>0\}$ and odd generators $\{e^{(r)},f^{(s)};~r>s_{1,2},s>s_{2,1}\}$ subject to all of the relations from Theorem \ref{thm: Drinfeld presentation} that make sense.
		Clearly, there is a homomorphism $\Y_\sigma\rightarrow \Y$ that sends the generators of $\Y$ to the generators with the same name in $\Y$. 
		Using the PBW theorem (Theorem \ref{theorem: PBW-drinfeld}), 
		we see that the homomorphism $\Y_\sigma\rightarrow \Y$ is injective (cf. \cite[Theorem 2.8]{BBG13} and \cite[Corollary 4.5]{Peng21}).
		Hence $\Y_\sigma$ can be identified as a subalgebra of $\Y$.
		
		Assume that $s_{1,2}>0$.
		Define $\sigma_+$ by
		\begin{equation*}
			\sigma_+:=
			\begin{pmatrix}
				0 & s_{1,2}-1 \\
				s_{2,1} & 0
			\end{pmatrix}.
		\end{equation*}
		Similarly, if $s_{2,1}>0$, then we define the shift matrix $\sigma_-$ by
		\begin{equation*}
			\sigma_-:=
			\begin{pmatrix}
				0 & s_{1,2} \\
				s_{2,1}-1 & 0
			\end{pmatrix}.
		\end{equation*}
		It follows from \cite[(2-14), (2-15)]{BBG13} that there exists well-defined algebra homomorphisms
		\begin{align}\label{delta +}
			\Delta_+:&\Y_\sigma\rightarrow \Y_{\sigma_+}\otimes \,U(\fgl_1)\\
			d_1^{(r)}\mapsto d_1^{(r)}\otimes 1,&\quad  d_2^{(r)}\mapsto d_2^{(r)}\otimes 1-d_{2}^{(r-1)}\otimes e_{1,1},\nonumber\\
			e^{(r)}\mapsto e^{(r)}\otimes 1-e^{(r-1)}\otimes e_{1,1},&\quad  f^{(r)}\mapsto f^{(r)}\otimes 1,\nonumber
		\end{align}
		and 
		\begin{align}\label{delta -}
			\Delta_-:&\Y_\sigma\rightarrow U(\fgl_1)\otimes \Y_{\sigma_-}\\
			d_1^{(r)}\mapsto 1\otimes d_1^{(r)},&\quad  d_2^{(r)}\mapsto 1\otimes d_2^{(r)}-e_{1,1}\otimes d_{2}^{(r-1)},\nonumber\\
			f^{(r)}\mapsto 1\otimes f^{(r)}-e_{1,1}\otimes f^{(r-1)},&\quad  e^{(r)}\mapsto 1\otimes e^{(r)},\nonumber
		\end{align}
		where $U(\fgl_1)$ is the universal enveloping algebra of the one-dimensional Lie algebra $\fgl_1$ with basis $e_{1,1}$.
		
		Let $\sigma=(s_{i,j})_{1\leq i,j\leq 2}$ be a shift matrix. 
		Suppose also that we are given an integer $\ell\geq s_{1,2}+s_{2,1}$, and set
		\begin{align}\label{definition of k}
			k:=\ell-s_{1,2}-s_{2,1}\geq 0.
		\end{align}
		By \cite[Lemma 2.1]{BBG13}, we can iterate $\Delta_+$ a total of $s_{1,2}$ times, $\Delta_-$ a total of $s_{2,1}$ times and the comultiplication $\Delta$ a total of $k-1$ times.
		There results a well-defined morphism 
		\[
		\Delta_{\sigma,\ell}:\Y_\sigma\rightarrow U(\fgl_1)^{\otimes s_{2,1}}\otimes \Y^{\otimes k}\otimes \,U(\fgl_1)^{\otimes s_{1,2}}.
		\]
		Let 
		\[
		U_{\sigma,\ell}:=U(\fgl_1)^{\otimes s_{2,1}}\otimes U(\fgl_{1|1})^{\otimes k}\otimes U(\fgl_1)^{\otimes s_{1,2}}. 
		\]
		Recalling \eqref{evaluation map}, we obtain a homomorphism
		\[
		\ev_{\sigma,\ell}:=(\id^{\otimes s_{2,1}}\otimes \ev^{\otimes k}\otimes \id^{\otimes s_{1,2}})\circ\Delta_{\sigma,\ell}:\Y_\sigma\rightarrow U_{\sigma,\ell}.
		\]
		We define the {\it shifted super Yangian of level $\ell$}
		\[
		\Y_{\sigma,\ell}:= \ev_{\sigma,\ell}(\Y_\sigma)\subseteq U_{\sigma,\ell}.
		\]
		Let $I_{\sigma,\ell}$ be the two-sided ideal of $\Y_\sigma$ generated by the elements $d_1^{(r)}$ for $r>k$.
		
		\begin{Theorem}
			The homomorphism $\ev_{\sigma,\ell}$ induces an algebra isomorphism $\Y_\sigma/I_{\sigma,\ell}\cong\Y_{\sigma,\ell}$.
			Moreover,
			ordered supermonomials in the elements
			\begin{align}\label{Omega definition}
				\Omega:=\{d_1^{(r)};~0<r\leq k\}\cup \{d_2^{(r)};&~0<r\leq \ell\}\\
				&\cup\{e^{(r)};~s_{1,2}<r\leq s_{1,2}+k\}\cup \{f^{(r)};~s_{2,1}<r\leq s_{2,1}+k\}\nonumber
			\end{align}
			taken in any fixed ordering form a basis for $\Y_{\sigma,\ell}$.
		\end{Theorem}
		\begin{proof}
			The proof given in characteristic zero in \cite[Theorem 3.6 and Corollary 3.6]{BBG13}, with the use of \cite[Theorem 1]{Gow07} (as stated in \cite[Theorem 3.1]{CH23}) works as well in positive characteristic.
		\end{proof}
		
		Henceforth, we will identify $\Y_{\sigma,\ell}$ with the quotient $\Y_\sigma/I_{\sigma,\ell}$, and we will abuse notation by denoting the canonical images in $\Y_{\sigma,\ell}$ of the elements $d_i^{(r)}, e^{(r)},\dots$ of $\Y_\sigma$ by the same symbols.
		Observe that the elements $\{b_i^{(rp)};~i=1,2,r>0\}$ \eqref{generators of p center} are contained in $\Y_\sigma$,  they are certainly central in the subalgebra $\Y_\sigma$. Also define the {\it $p$-center} $Z_p(\Y_\sigma)$ of $\Y_\sigma$ to be the subalgebra generated by 
		\[
		\{b_i^{(rp)};~i=1,2,r>0\}.
		\]
		Moreover, we define the {\it $p$-center} of $\Y_{\sigma,\ell}$ to be the image of the $p$-center of $\Y_\sigma$ in $\Y_{\sigma,\ell}$, and denote it by $Z_p(\Y_{\sigma,\ell})$.
		We let $Z_p(\Y_{\sigma,\ell})_+$ to be the ideal of $Z_p(\Y_{\sigma,\ell})$ generated by the elements $\{b_i^{(r)};~i=1,2, r>0\}$ (see Remark \ref{Remark: generators of p-center}) and define the {\it restricted truncated shifted super Yangian} 
		\[
		\Y_{\sigma,\ell}^{[p]}:=\Y_{\sigma,\ell}/\Y_{\sigma,\ell}Z_p(\Y_{\sigma,\ell})_+.
		\]
		\subsection{Irreducible representations of $\Y_{\sigma,\ell}^{[p]}$}
		We continue using the notations of section \ref{section name:Shifted super Yangains}.
		Then one can obtain a two-rowed {\it pyramid} $\pi$ from the tuple $(\sigma,\ell)$.
		The first row of $\pi$ consists of $k$ \eqref{definition of k} boxes and the section row consists of $\ell$ boxes.
		If $k>0$, then $\pi$ has $s_{2,1}$ columns of height $1$ on its left side and $s_{1,2}$ columns of height $1$ on its right side.
		For example, let
		\begin{equation*}
			\sigma:=
			\begin{pmatrix}
				0 & 2 \\
				1& 0
			\end{pmatrix}
		\end{equation*}
		and take $\ell=5$, the resulted pyramid $\pi$ is
		\begin{equation*}
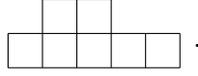
			\begin{array}{c}
				\begin{picture}(65,26)
					\put(0,0){\line(1,0){65}}
					\put(0,13){\line(1,0){65}} 
					\put(13,26){\line(1,0){26}}
					\put(0,0){\line(0,1){13}} 
					\put(13,0){\line(0,1){26}}
					\put(26,0){\line(0,1){26}} 
					\put(39,0){\line(0,1){26}}
					\put(52,0){\line(0,1){13}}
					\put(65,0){\line(0,1){13}}
				\end{picture}
			\end{array}.
		\end{equation*}
		Conversely, given a two-rowed pyramid $\pi$, we are able to recover the shift matrix $\sigma$ and $\ell$ (cf. \cite[Proposition 2.8]{Peng21}).
		
		A {\it $\pi$-tableau} is a diagram obtained by filling the boxes of $\pi$ with elements of $\kk$. 
		The set of all tableaux of shape $\pi$ is denoted $\Tab_{\kk}(\pi)$.
		We represent the $\pi$-tableau with the entries $a_1,\dots,a_k$ along its first row and $b_1,\dots,b_{\ell}$ along its second row simply by the array
		$\begin{tiny}
			\begin{array}{ccc}
				a_1\cdots a_k\\
				b_1\cdots b_{\ell}
			\end{array}
		\end{tiny}$.
		We say that $A,B\in \Tab_{\kk}(\pi)$ are {\it row equivalent}, denoted $A\sim B$, if $B$ can be obtained from $A$ by permuting the entries in the rows.
		
		We will write $\Y_\pi$ (resp. $\Y_\pi^{[p]}$) instead of $\Y_{\sigma,\ell}$ (resp. $\Y_{\sigma,\ell}^{[p]}$) to ease notation.
		By the relations, $\Y_\pi$ admits a $\ZZ$-grading (cf. \cite[Section 6]{BBG13}\footnote{The authors mainly considered the principal $W$-algebra $W_\pi$ for $\fgl_{m|n}$ in {\it loc. cit.}
			They identify $\Y_{\sigma,\ell}$ with $W_\pi$ via the isomorphism in \cite[Theorem 4.5]{BBG13}.})
		\begin{align}\label{z-grading}
			\Y_\pi=\bigoplus\limits_{g\in\ZZ}\Y_{\pi;g}
		\end{align}
		such that the generators $d_i^{(r)}$ are of degree $0$,
		the generators $e^{(r)}$ are of degree $1$ and the generators $f^{(r)}$ are of degree $-1$.
		Moreover, the algebra $\Y_\pi$ admits a triangular decomposition. 
		We first introduce some subsets of $\Omega$ \eqref{Omega definition}.
		Put
		\[
		\Omega_0:=\{d_1^{(r)}, d_2^{(s)} ;~0<r\leq k, 0<s\leq \ell\}
		\]
		and 
		\[
		\Omega_+:=\{e^{(r)};~s_{1,2}<r\leq s_{1,2}+k\} ,\quad  \Omega_-:=\{f^{(r)};~s_{2,1}<r\leq s_{2,1}+k\},
		\]
		and let $\Y_\pi^0$, $\Y_{\pi}^{-}$ and $\Y_\pi^\sharp$ be the subalgebras of $\Y_\pi$ generated by $\Omega_0$, $\Omega_-$ and $\Omega_0\cup\Omega_+$, respectively.
		
		For $A=\!\!\begin{tiny}
			\begin{array}{ccc}
				a_1\cdots a_k\\
				b_1\cdots b_{\ell}
			\end{array}
		\end{tiny}\!\!\in \Tab_{\kk}(\pi)$, let $\kk_A$ be the one-dimensional $\Y_\pi^0$-module on basis $1_A$ such that
		\begin{align}\label{ukd1u}
			u^kd_1(u).1_A=(u+a_1)\cdots(u+a_k)1_A,
		\end{align}
		\begin{align}\label{uld2u}
			u^{\ell}d_2(u).1_A=(u+b_1)\cdots(u+b_{\ell})1_A.
		\end{align}
		Thus, $d_1^{(r)}.1_A=e_r(a_1,\dots,a_k)1_A$ and $d_2^{(r)}.1_A=e_r(b_1,\dots,b_l)1_A$,
		where $e_r$ denotes the $r$-th elementary symmetric polynomial.
		Since $\Y_\pi^0$ is the polynomial algebra on $\Omega_0$,
		every irreducible $\Y_\pi^0$ is isomorphic to $\kk_A$ for some $A\in\Tab_{\kk}(\pi)$, and $\kk_A\cong\kk_B$ if and only if $A\sim B$.
		
		Thanks to \cite[Theorem 6.1]{BBG13}, there is a unique surjective homormorphism 
		$\Y_\pi^\sharp\twoheadrightarrow \Y_\pi^0$ sending $e^{(r)}\mapsto 0$ for all $r>s_{1,2}$. We can view $\kk_A$ as a $\Y_\pi^\sharp$-module via pullback along the surjection. We define the {\it Verma module}
		\begin{align}\label{verma module for y pi}
			M(A):=\Y_\pi\otimes_{\Y_\pi^\sharp}\kk_A.    
		\end{align}
		From the PBW theorem, $\Y_\pi$ is a free right $\Y_\pi^\sharp$-module with basis given by the ordered supermonomials in $\Omega_-$. 
		This yields $\dim M(A)=2^k$.
		
		The following lemma in characteristic zero is given in \cite[Lemma 7.1]{BBG13}.
		In our situation, we are going to employ the similar proof of Proposition \ref{prop: zpl unique max submodule}.
		\begin{Lemma}\label{lemma:MA has unique max submodule}
			For $A=\!\!\begin{tiny}
				\begin{array}{ccc}
					a_1\cdots a_k\\
					b_1\cdots b_{\ell}
				\end{array}
			\end{tiny}\!\!\in \Tab_{\kk}(\pi)$, 
			the Verma module $M(A)$ has a unique irreducible quotient $L(A)$.
			The image $v_+$ of $1\otimes 1_A$ is the unique  (up to scalars) nonzero vector in $L(A)$ such that $e^{(r)}.v_+ = 0 $ for all $r>s_{1,2}$.
			Moreover, we have that $ d_1^{(r)}.v_+ = e_r(a_1, \dots, a_k)v_+$ and $ d_2^{(s)}.v_+ = e_r(b_1,\dots, b_{\ell})v_+ $ for all $r\geq 0$.
		\end{Lemma}
		\begin{proof}
			We know that
			\[
			X:=\{(f^{(s_{2,1}+1)})^{i_1}(f^{(s_{2,1}+2)})^{i_2}\cdots(f^{(s_{2,1}+k)})^{i_k};~i_j\in \{0,1\}, j=1,\dots, k\}.
			\]
			is a basis of $\Y_{\pi}^{-}$ and $\Y_{\pi}^{-}\otimes 1_A\cong M(A)$.
			Write 
			\[f_{i_1,\dots,i_k}:=(f^{(s_{2,1}+1)})^{i_1}(f^{(s_{2,1}+2)})^{i_2}\cdots(f^{(s_{2,1}+k)})^{i_k}
			\]
			for short and define
			\[
			(\Y_{\pi}^{-})_+:=\Span_\kk\{f_{i_1,\dots,i_k}\in X;~(i_1,\dots,i_k)\neq (0,\dots,0)\}.
			\]
			It suffices to prove that any proper submodule $M$ of $M(A)$ is contained in $(\Y_{\pi}^{-})_+\otimes 1_A$.
			For any $x\in (\Y_{\pi}^{-})_+$, we observe that $x$ is nilpotent.
			Then we may apply precisely the same argument as for Proposition \ref{prop: zpl unique max submodule} to get that $M(A)$ has a unique simple quotient $L(A)$.
			
			The foregoing observations in conjunction with \eqref{ukd1u} and \eqref{uld2u} imply that $v_+$ is a nonzero vector which satisfies our requirements.
			It remains to show that any vector $v\in L(A)$ annihilated by all $e^{(r)}$ is  proportional to $v_+$.
			Using the $\ZZ$-grading of $\Y_\pi$ \eqref{z-grading},
			we know that
			\[
			L(A)=\Y_\pi.v_+=\sum\limits_{g\in\ZZ}\Y_{\pi;g}.v_+.
			\]
			We fix the order of generators of $\Y_\pi$ such that $\Omega_{-}<\Omega_0<\Omega_+$.
			Since $e^{(r)}.v_+=0$,
			the PBW theorem implies that $\Y_{\pi;0}.v_+=\kk v_+$ and $ \Y_{\pi;g}.v_+=(0)$ for $g>0$.
			It follows that 
			\[
			L(A)=\kk v_+ +\sum\limits_{g<0}\Y_{\pi;g}.v_+=\kk v_+ +(\Y_{\pi}^{-})_+.v_+.
			\]
			Recall that $(\Y_{\pi}^{-})_+$ consists of nilpotent elements,
			hence, $L(A)=\kk v_+ \oplus (\Y_{\pi}^{-})_+.v_+$.
			Write $v=a_0v_+ + x.v_+$ with $a_0\in\kk$ and $x\in (\Y_{\pi}^{-})_+$.
			We need to show that $x.v_+=0$.
			Next, we observe that 
			\[
			0=e^{(r)}.v=e^{(r)}.(a_0v_+ + x.v_+)=e^{(r)}x.v_+   
			\] and $x.v_+\in\sum_{g<0}\Y_{\pi;g}.v_+$.
			This means the submodule $\Y_\pi.(x.v_+)\subseteq\sum_{g<0}\Y_{\pi;g}.v_+=(\Y_{\pi}^{-})_+.v_+$,
			so it must be $0$.
		\end{proof}
		
		\begin{Theorem}\label{thm:classification of simple ypimodule}
			Every simple $\Y_\pi$-module is finite-dimensional and is isomorphic to one of the modules $L(A)$ from Lemma \ref{lemma:MA has unique max submodule} for some $A\in\Tab_{\kk}(\pi)$. 
			Moreover, $L(A)\cong L(B)$ if and only if $ A \sim B$.
			Hence, fixing a set $\Tab_{\kk}(\pi)/\sim $ of representatives for the $ \sim$-equivalence classes in $\Tab_{\kk}(\pi)$, the modules
			\[
			\{L(A);~A \in \Tab_{\kk}(\pi)/\sim \}
			\]
			give a complete set of pairwise inequivalent simple $\Y_\pi$-modules.
		\end{Theorem}
		\begin{proof}
			Given $A,B\in \Tab_{\kk}(\pi)$, 
			we apply Lemma \ref{lemma:MA has unique max submodule} to see that $L(A)\cong L(B)$ if and only if $A\sim B$.
			The rest of the proof relies on the $\ZZ$-grading of $\Y_\pi$ \eqref{z-grading} which works just as well in our situation.
			Then applying verbatim the proof of \cite[Theorem 7.2]{BBG13}
			gives the assertions.
		\end{proof}
		
		Recall that the pyramid $\pi$ has $\ell$ boxes on its second row.
		For each $c=1,\dots,\ell$, let $q_c$ be the height of the $c$-th column in the pyramid $\pi$.
		Given $A\in\Tab_{\kk}(\pi)$, let $A_c$ be its $c$-th column.
		If $q_c=1$, then $A_c$ has just a single entry $b$.
		We let $L(A_c)$ be the one-dimensional $U(\fgl_1)$-module with an even basis vector $v_+$ such that $e_{1,1}.v_+=-bv_+$.
		If $q_c=2$, then $A_c$ has two entries, $a$ in the first row and $b$ in the second row.
		In this case,
		let $L(A_c)$ be the irreducible $U(\fgl_{1|1})$-module with an even basis vector $v_+$ such that $e_{1,1}.v_+=av_+, e_{2,2}.v_+=-bv_+$ and $e_{1,2}.v_+=0$.
		Then $L(A_c)$ is one- or two-dimensional according to whether $a=b$ (see also \eqref{e12e21xi}).
		
		We consider the tensor product 
		\[
		L(A_1)\otimes\cdots\otimes L(A_\ell)
		\]
		which can be naturally thought of as a $U_{\sigma,\ell}$-module.
		Since $\Y_\sigma\stackrel{\ev_{\sigma,\ell}}{\twoheadrightarrow}\Y_\pi\subseteq U_{\sigma,\ell}$, 
		we may restrict it to a $\Y_\pi$-module.
		Then we obtain a well-defined $\Y_\pi$-module
		\begin{align}\label{definition V(A)}  
			V(A):=L(A_1)\otimes\cdots\otimes L(A_\ell).
		\end{align}
		Note that $\dim V(A)=2^{k-h}$,
		where $h$ is the number of $c=1,\dots,\ell$ such that $A_c$ has two equal entries.
		
		We write $v:=v_+\otimes\cdots\otimes v_+\in V(A)$.
		The following lemma shows that $v$ is a highest weight vector in $V(A)$ with the same weight as $1_A$ (see \eqref{ukd1u} and \eqref{uld2u}).
		This is the modular version of \cite[Lemma 8.3]{BBG13}.
		Since we don't use the $W$-algebra,
		we shall prove it by a direct induction.
		\begin{Lemma}\label{lemma: ma to va}
			For any $A\in\Tab_{\kk}(\pi)$,
			there is a nonzero homormorphism
			\[
			M(A)\rightarrow V(A)
			\]
			sending the cyclic vector $1\otimes 1_A\in M(A)$ to $v\in V(A)$.
			In particular,
			$V(A)$ contains a subquotient isomorphic to $L(A)$.
		\end{Lemma}
		\begin{proof}
			Suppose that $A=\!\!\begin{tiny}
				\begin{array}{ccc}
					a_1\cdots a_k\\
					b_1\cdots b_{\ell}
				\end{array}
			\end{tiny}\!\!\in \Tab_{\kk}(\pi)$.
			By the universal property of $M(A)$,
			it suffices to show that $v$ is annihilated by all $e^{(r)}$ for $r>s_{1,2}$ and that $d_1^{(r)}.v=e_r(a_1,\dots,a_k)v$ and $d_2^{(r)}.v=e_r(b_1,\dots,b_{\ell})v$ for all $r>0$.
			
			First we treat the case where $\sigma=(0)$.
			In order to ease notation,
			we put in the following
			\[
			\bar{e}_{i,j}^{[c]}:=(-1)^{|i|}1^{\otimes (c-1)}\otimes e_{i,j}\otimes 1^{\otimes(\ell-c)}  
			\]
			for any $1\leq i,j\leq 2$ and $1\leq c\leq \ell$.
			In view of \eqref{t11t22 drinfeld} and \eqref{t12 t21 drinfeld},
			it is equivalent to show that
			$t_{1,2}^{(r)}.v=0$ and 
			\[
			t_{1,1}^{(r)}.v=e_r(a_1,\dots,a_k)v,~t_{2,2}^{(r)}.v=e_r(b_1,\dots,b_\ell)v
			\]
			for all $r>0$.
			According to \cite[(3-4)]{BBG13} (see also \cite[Section 2]{Gow07}), 
			we have that 
			\[
			t_{i,j}^{(r)}.v=\sum\limits_{1\leq c_1<\cdots<c_r\leq \ell}\sum\limits_{1\leq h_1,\dots,h_{r-1}\leq 2}\bar{e}_{i,h_1}^{[c_1]}\bar{e}_{h_1,h_2}^{[c_2]}\cdots \bar{e}_{h_{r-1},j}^{[c_r]}.v.
			\]
			If $(i,j)=(1,2)$, then every summand in the above sum is zero because it contains a factor of the form $\bar{e}_{1,2}^{[c_m]}.v_+$ with $1\leq m\leq \ell$, which is zero.
			Similarly, if $i=j$,
			then the only nonzero summand corresponds to the case where each index $h_m$ equals $i$.
			Thus, the summation reduces just to 
			\[
			t_{i,i}^{(r)}.v=\sum\limits_{1\leq c_1<\cdots<c_r\leq \ell}\bar{e}_{i,i}^{[c_1]}\bar{e}_{i,i}^{[c_2]}\cdots \bar{e}_{i,i}^{[c_r]}.v,
			\]
			so that our assertions follow from an easy straightforward computation.
			
			Assume that $\sigma\neq (0)$.
			We proceed by induction on $\ell$.
			The case $\ell=1$ is clear. 
			If $s_{1,2}>0$,
			then $\Delta_+$ \eqref{delta +} induces a map $\Y_{\sigma,\ell}\rightarrow \Y_{\sigma_+,\ell-1}\otimes\,U(\fgl_1)$ (cf. \cite[Remark 8.2]{BBG13}).
			We write $v':=v_+\otimes \cdots \otimes v_+$ for the tensor product of the first $\ell-1$ maximal vectors. Then $v=v'\otimes v_+$.
			We conclude from \eqref{delta +} and the induction hypothesis
			\begin{align*}
				d_1^{(r)}.v=&d_1^{(r)}.v'\otimes v_+=e_r(a_1,\dots,a_k)v'\otimes v_+=e_r(a_1,\dots,a_k)v\\
				d_2^{(r)}.v=&d_2^{(r)}.v'\otimes v_+-d_2^{(r-1)}.v'\otimes e_{1,1}.v_+\\ 
				=&e_r(b_1,\dots,b_{\ell-1})v-e_{r-1}(b_1,\dots,b_{\ell-1})\cdot (-b_{\ell})v\\
				=&e_r(b_1\dots,b_\ell)v
			\end{align*}
			and 
			\[
			e^{(r)}.v=e^{(r)}.v'\otimes v_+-e^{(r-1)}.v'\otimes e_{1,1}.v_+=0.      
			\]
			One argues similarly for the case $s_{2,1}>0$ using \eqref{delta -} instead of \eqref{delta +}.
			Then we can reduce the proof to the case $\sigma=(0)$ by using the maps $\Delta_+$ and $\Delta_{-}$.
		\end{proof}
		
		Now we state the modular version of \cite[Theorem 8.4]{BBG13}.
		\begin{Theorem}\label{main theorem 1}
			Take any $A\in\Tab_{\kk}(\pi)$,
			and let $h\geq 0$ be maximal such that distinct $1\leq i_1,\dots,i_h\leq k$
			and $1\leq j_1,\dots,j_h\leq \ell$ with $a_{i_1}=b_{j_1},\dots,a_{i_h}=b_{j_h}$ exist.
			Choose $B\sim A$ such that $B$ has $h$ columns of height $2$ containing equal entries.
			Then
			\[
			L(A)\cong V(B).
			\]
			In particular, $\dim L(A)=2^{k-h}$.
		\end{Theorem}
		\begin{proof}
			The proof, which uses Lemma \ref{lemma: ma to va} and \cite[Theorem 7.3]{BBG13},
			is similar to the proof of \cite[Theorem 8.4]{BBG13}.
			Notice that the proof of \cite[Theorem 7.3]{BBG13} works as well in positive characteristic.
		\end{proof}
		
		Now we turn to the irreducible $\Y_\pi^{[p]}$-modules.
		Recall from Section \ref{section name:Shifted super Yangains} that
		\[
		\Y_\sigma\stackrel{\ev_{\sigma,\ell}}{\twoheadrightarrow}\Y_\pi\twoheadrightarrow \Y_\pi^{[p]}:=\Y_{\sigma,\ell}/\Y_{\sigma,\ell}Z_p(\Y_{\sigma,\ell})_+.
		\]
		By Theorem \ref{thm:classification of simple ypimodule}
		we have a classification of simple $\Y_\pi$-modules given by the
		modules $L(A)$ for $A\in\Tab_{\kk}(\pi)$ ranging over a set of representatives of row equivalence classes of $\pi$-tableaux.
		We shall determine for which of these simple modules the action of $\Y_\pi$ factors through the quotient $\Y_\pi\twoheadrightarrow \Y_\pi^{[p]}$ to give an irreducible $\Y_\pi^{[p]}$-module.
		
		Given $A=\!\!\begin{tiny}
			\begin{array}{ccc}
				a_1\cdots a_k\\
				b_1\cdots b_{\ell}
			\end{array}
		\end{tiny}\!\!\in \Tab_{\kk}(\pi)$,
		we let $L(A)$ be an irreducible $\Y_\pi$-module from Theorem \ref{thm:classification of simple ypimodule}.
		Denote by $v_+$ the unique maximal vector of $L(A)$ (see Lemma \ref{lemma:MA has unique max submodule}).
		We note that the module $L(A)$ factors to a module for $\Y_\pi^{[p]}$ if and only if $b_i^{(r)}.v_+=0$ for $i=1,2$.
		In view of \eqref{def of biu} and \eqref{ukd1u},
		we have 
		\begin{align*}
			b_1(u).v_+=&d_1(u)d_1(u-1)\cdots d_1(u-p+1).v_+\\
			=&\frac{\prod\limits_{i=1}^k \prod\limits_{j=0}^{p-1}(u+a_i-j)}{(\prod\limits_{j=0}^{p-1}(u-j))^k}v_+.
		\end{align*}
		As a result, $b_1^{(r)}.v_+=0$ if and only if $a_i\in \FF_p$ all $1\leq i\leq k$.
		Similarly, one obtains from \eqref{uld2u} that
		$b_2^{(r)}.v_+=0$ if and only if $b_j\in \FF_p$ all $1\leq j\leq \ell$.
		
		We write $\Tab_{\FF_p}(\pi)\subseteq \Tab_{\kk}(\pi)$ for those tableaux with entries in $\FF_p$.
		Taking into account Theorem \ref{thm:classification of simple ypimodule},
		we readily obtain:
		\begin{Theorem}\label{thm:classification of simple restricted ypi-module}
			Every simple $\Y_\pi^{[p]}$-module is finite-dimensional and is isomorphic to one of the modules $L(A)$ from Lemma \ref{lemma:MA has unique max submodule} for some $A\in\Tab_{\FF_p}(\pi)$. 
			Moreover, $L(A)\cong L(B)$ if and only if $ A \sim B$.
			Hence, fixing a set $\Tab_{\FF_p}(\pi)/\sim $ of representatives for the $ \sim$-equivalence classes in $\Tab_{\FF_p}(\pi)$, the modules
			\[
			\{L(A);~A \in \Tab_{\FF_p}(\pi)/\sim \}
			\]
			give a complete set of pairwise inequivalent simple $\Y_\pi^{[p]}$-modules.
		\end{Theorem}
		
		\begin{Remark}\label{remark last}
			In \cite{BBG13},
			the authors proved that there is a superalgerba isomorphism $\mu: W_\pi\cong \Y_\pi$,
			where $W_\pi$ is the principal $W$-algebra associated to the Lie superalgebra $\fgl_{M|N}$, see Theorem 4.5 of {\it loc. cit.} for details.
			In joint work with Hu and Topley \cite{CHT25} the first author studied the certain modular representations of the general linear Lie algebra $\fgl_{2n}$
			by employing the restricted version of Brundan–Kleshchev’s isomorphism (see \cite[Theorem 1.1]{GT21}, \cite[Theorem 10.1]{BK06}).
			We expect the isomorphism $\mu$ factor through the restricted quotients,
			i.e., $W_\pi^{[p]}\cong \Y_\pi^{[p]}$.
			Here, $W_\pi^{[p]}$ is the restricted principal $W$-superalgebra (cf. \cite[Theorem 4.4]{WZ09}, \cite[Theorem 2.24]{ZS15}).
			In conjunction with Theorem \ref{thm:classification of simple restricted ypi-module} this will lead to a classification of irreducible modules of the reduced enveloping algebra of $\fgl_{m|n}$ associated to the regular nilpotent $p$-character (cf. \cite{WZ09}).

		\end{Remark}
		\bigskip
		\appendix
		\section{Proof of Proposition \ref{prop: delta send p to p}}\label{section name:appendix}
		In the appendix, 
		we provide a detailed proof for Proposition \ref{prop: delta send p to p}.
		The result is rather parallel to that of the Yangian $Y_2$ (see \cite[Section 2.3]{CHT25}).
		
		\begin{Lemma}\label{lemma A-1}
			The following identities hold in $\Y[[u^{-1}]]$:  
			\begin{itemize}
				\item[(1)] $t_{1,1}(u-1)t_{1,2}(u)=t_{1,2}(u-1)t_{1,1}(u)$;
				\item[(2)] $t_{1,1}(u)t_{2,1}(u-1)=t_{2,1}(u)t_{1,1}(u-1)$;
				\item[(3)] $t_{1,2}(u-1)t_{1,2}(u)=0$;
				\item[(4)] $(k+1)t_{2,1}(u)t_{1,1}(u-k)=kt_{1,1}(u-k)t_{2,1}(u)+t_{2,1}(u-k)t_{1,1}(u)$~\text{for}~$k\in\ZZ_{\geq 0}$.
			\end{itemize}
		\end{Lemma}
		\begin{proof}
			These follow from the defining relation \eqref{def relation-series}.
			Using the anti-automorphism $\tau$ \eqref{tau-antiauto},
			we see that the first three equations follow immediately from \cite[(4.4)-(4.6)]{CH23-2}.
			To get (4), set $(i,j,k,l)=(1,1,2,1)$ and $u:=v-k$ in \eqref{def relation-series}, simplify, then replace $v$ by $u$.
		\end{proof}
		\begin{Lemma}\label{lemma A-2}
			For any positive integer $n$, we have that
			\begin{align*}
				\Delta(t_{1,1}(u-n+1)&t_{1,1}(u-n+2)\cdots t_{1,1}(u))\\
				=t_{1,1}&(u-n+1)t_{1,1}(u-n+2)\cdots t_{1,1}(u)\otimes t_{1,1}(u)\cdots t_{1,1}(u-n+2)t_{1,1}(u-n+1)\\
				&+n t_{1,2}(u-n+1)t_{1,1}(u-n+2)\cdots t_{1,1}(u)\otimes t_{1,1}(u)\cdots t_{1,1}(u-n+2)t_{2,1}(u-n+1).
			\end{align*}
		\end{Lemma}
		\begin{proof}
			We proceed by induction on $n$.
			For $n=1$, this holds by the definition of $\Delta$ \eqref{comul and antipode}.
			So assume that $n\geq 2$ and the statement is true for $n-1$. 
			It follows that
			\begin{align*}
				\Delta\big(t_{1,1}(u-n+1)t_{1,1}(u-n+2)\cdots &t_{1,1}(u)\big)\\ 
				=\,\,\Delta\big(t_{1,1}(u-n+1)\big)\Delta\big(t_{1,1}&(u-n+2)\cdots t_{1,1}(u)\big)\\
				=\,\,\big(t_{1,1}(u-n+1)\otimes&t_{1,1}(u-n+1)+t_{1,2}(u-n+1)\otimes t_{2,1}(u-n+1)\big)\\
				\times\big(t_{1,1}(u-n+2)&\cdots t_{1,1}(u)\otimes t_{1,1}(u)\cdots t_{1,1}(u-n+2)\\
				+(n-1)t_{1,2}&(u-n+2)t_{1,1}(u-n+3)\cdots t_{1,1}(u)\\
				\otimes & t_{1,1}(u)\cdots t_{1,1}(u-n+3)t_{2,1}(u-n+2)\big).
			\end{align*}
			Using Lemma \ref{lemma A-1}(3) and simplifying the above,
			we only need to show that
			\begin{align*}
				n t_{1,2}(u-n+1)t_{1,1}(u-n+2)&\cdots t_{1,1}(u)\otimes t_{1,1}(u)\cdots t_{1,1}(u-n+2)t_{2,1}(u-n+1)\\
				=t_{1,2}(u-n+1)t_{1,1}(u-n+2)&\cdots t_{1,1}(u)\otimes t_{2,1}(u-n+1)t_{1,1}(u)\cdots t_{1,1}(u-n+2)\\
				+ (n-1)t_{1,1}(u-n+1)&t_{1,2}(u-n+2)t_{1,1}(u-n+3)\cdots t_{1,1}(u)\\
				\otimes t_{1,1}(u-n+1)&t_{1,1}(u)\cdots t_{1,1}(u-n+3)t_{2,1}(u-n+2).
			\end{align*}
			Moreover,  
			Lemma \ref{lemma A-1}(1) implies that all the left sides of the above tensor products are equal.
			It remains to check that 
			\begin{align*}
				nt_{1,1}(u)&\cdots t_{1,1}(u-n+2)t_{2,1}(u-n+1)\\
				=t_{2,1}&(u-n+1)t_{1,1}(u)\cdots t_{1,1}(u-n+2)\\
				&+ (n-1) t_{1,1}(u-n+1)t_{1,1}(u)\cdots t_{1,1}(u-n+3)t_{2,1}(u-n+2).
			\end{align*}
			This follows from Lemma \ref{lemma A-1}(1) and (4) taking $k:=n-1$.
		\end{proof}
		\begin{Proposition}{(=Proposition \ref{prop: delta send p to p})}
			We have
			\begin{align*}
				\Delta(b_i(u))=b_i(u)\otimes b_i(u).
			\end{align*}
			In particular, $\Delta(I_p)\subseteq I_p \otimes \Y+\Y\otimes I_p$.
		\end{Proposition}
		\begin{proof}
			Recall from \eqref{t11t22 drinfeld} that
			$t_{1,1}(u)=d_1(u)$ ,
			so that $b_1(u)=t_{1,1}(u)\cdots t_{1,1}(u-p+1)$.
			Taking in Lemma \ref{lemma A-2} $n:=p$ gives
			\[
			\Delta(b_1(u))=b_1(u)\otimes b_1(u). 
			\]
			We now consider the case $i=2$.
			In view of \cite[Proposition 1]{Gow07},
			the map
			\[
			\zeta:\Y\rightarrow \Y;~t_{i,j}(u)\mapsto t'_{3-i,3-j}(u)
			\]
			gives an algebra isomorphism.
			It is easy to verify that $\zeta(t'_{i,j}(u))=t_{3-i,3-j}(u)$.
			Moreover,
			let $P:\Y\otimes \Y\rightarrow \Y\otimes \Y$ be the map given by 
			\[
			P(x\otimes y)=y\otimes x (-1)^{|x||y|}
			\]
			for all homogeneous elements $x,y\in\Y$.
			Then we have $\Delta=(\zeta\otimes \zeta)\circ P\circ \Delta\circ\zeta$ (cf. \cite[Proposition 2]{Gow07}).
			As $t'_{2,2}(u)=d_2(u)^{-1}$ (cf. \cite[(14)]{Gow07}),
			we obtain
			\begin{align*}
				\Delta(b_2(u))=&\Delta(t'_{2,2}(u)\cdots t'_{2,2}(u-p+1))\\
				=&(\zeta\otimes \zeta)\circ P\circ \Delta\circ\zeta(t'_{2,2}(u)\cdots t'_{2,2}(u-p+1))\\
				=&(\zeta\otimes \zeta)\circ P\circ \Delta (b_1(u)).
			\end{align*}
			Remember that we already proved $\Delta(b_1(u))=b_1(u)\otimes b_1(u)$.
			Observing that $\zeta(b_1(u))=b_2(u)$,
			we thus obtain $\Delta(b_2(u))=b_2(u)\otimes b_2(u)$,
			as desired.
			
		\end{proof}
		
		\bigskip
		\noindent
		\textbf{Acknowledgment.}
		This work is supported by the Natural Science Foundation of Hubei Province (No. 2025AFB716).
		
	\end{document}